\documentclass{amsart}

\newtheorem{theorem}[equation]{Theorem}
\newtheorem{lemma}[equation]{Lemma}
\newtheorem{corollary}[equation]{Corollary}
\newtheorem{proposition}[equation]{Proposition}

\numberwithin{equation}{section}

\usepackage{amscd,amsmath,amssymb,verbatim}

\begin{document}

\title[On the integrality of hypergeometric series]{On the integrality of hypergeometric series whose coefficients are factorial ratios}
\author{Alan Adolphson}
\address{Department of Mathematics\\
Oklahoma State University\\
Stillwater, Oklahoma 74078}
\email{adolphs@math.okstate.edu}
\author{Steven Sperber}
\address{School of Mathematics\\
University of Minnesota\\
Minneapolis, Minnesota 55455}
\email{sperber@math.umn.edu}
\date{\today}
\begin{abstract}
We use the Dwork-Frobenius operator to prove an integrality result for $A$-hypergeometric series whose coefficients are factorial ratios.  As a special case, we generalize one direction of a classical result of Landau on the integrality of factorial ratios.
\end{abstract}
\maketitle

\section{Introduction}

Integrality properties for the coefficients of hypergeometric series appear in the literature
as early as 1900.  Specifically, Landau\cite{L} considered certain ratios of factorials (which can be viewed as coefficients of a suitable hypergeometric series: see Example~2 below) and gave a criterion for these ratios to be integral. 

In subsequent work Dwork\cite{D1,D2} (and then a number of other researchers focussing on the $p$-adic theory of differential equations, including Christol\cite{C} and the authors\cite{AS1, AS3}) was concerned with classifying hypergeometric series whose coefficients are $p$-integral. Such issues are in a sense more general than Landau's work and also have a different motivation. For Dwork, the key analytic property of the Frobenius map, say, on relative cohomology of a family of varieties defined over a finite field, is its analytic continuation to a $p$-adic lift of the Hasse domain (and even into the supersingular disks:\cite{D1,D3}). Once established, Frobenius acts on the solutions of the associated $p$-adic differential equation and important arithmetic information can be extracted.  In an early attempt to systematize an approach to proving analytic continuation to this large region, Dwork undertook a very detailed examination, finding precise $p$-adic estimates for ratios of  ``binomial-type numbers'', i.~e.,  ratios of rising Pochhammer factorials of the sort appearing in hypergeometric series. These ``formal congruences'' were used in an essential manner by Delaygue \cite{De1,De2}, Delaygue-Rivoal-Roques\cite{DRR}, and Krattenthaler-Rivoal\cite{KR} in their proofs of the integrality of the mirror map in cases of hypergeometric differential equations having a series solution with integral coefficients and having as well a log solution. 

In \cite{D1.5,D3} Dwork developed other methods for proving the analytic continuation and integrality.  He applied his construction of a Frobenius action on the solution matrix of the deformation equation associated to a family of varieties to conclude that certain distinguished solutions were stable under Frobenius. Using the strong estimates of Frobenius and the classical fixed point theorem for contraction mappings on a ($p$-adic) Banach space he could then prove that important ratios of hypergeometric functions had analytic continuation and the series themselves had $p$-integral coefficients. 

In general, in this setting, the Frobenius action on solutions in a singular disk can be used to obtain arithmetic information at smooth fibers. The information in question concerns the $p$-adic size of roots of the zeta or $L$-function at non-singular fibers and $p$-adic formulas for the roots themselves. In particular, the eigenvalues of minimal $p$-divisibility often can be expressed in terms of the analytic continuation of a certain ratio of the special function solution in the singular disk (for some recent results of this type see \cite{AS2}, Yu\cite{Y}, Zhang\cite{Z}). Even the Hasse invariant itself can be expressed in terms of a ``ratio'' of  matrices of suitable series solutions at a singular point.  An essential ingredient in so doing is knowing the coefficients of the series solutions under consideration are themselves $p$-integral. We expect the $p$-integrality results derived here to enable us in future work to establish $p$-adic formulas for Hasse invariants and to extend some known integrality results for mirror maps. 

In previous articles (\cite{AS1,AS3}) we dealt with confluent and non-confluent cases and used a recursive truncation method to obtain $p$-integrality results. Here we proceed somewhat differently.   We directly identify certain series whose coefficients are themselves hypergeometric series as eigenvectors of the Dwork-Frobenius operator.  If the corresponding eigenvalues are not too highly divisible by $p$ then the coefficients of these hypergeometric series are all $p$-integral.  Our hypothesis guarantees that this holds for all primes $p$, so in fact these hypergeometric series have integral coefficients (see Theorem~2.10).  In particular, we avoid computing the full solution matrix of the deformation equation and avoid finding $p$-adic estimates for the full Dwork-Frobenius matrix.
 
Our results here should play an essential role in giving a $p$-adic interpolation of the generalized Hasse invariant (see \cite{AS2.5}) in terms of a matrix of special functions.  We also indicate in the examples how our results give an extension of Landau's classical theorem on the integrality of factorial ratios.

\section{Results}

We describe the hypergeometric series we are considering.  Let $A = \{{\bf a}_j\}_{j=1}^N$ with ${\bf a}_j = (a_{1j},a_{2j},\dots,a_{n-1,j},1)\in{\mathbb Z}^n$ for all $j$.  Let $A' = \{{\bf a}_j\}_{j=1}^M\subseteq A$ and put
$\beta = \sum_{j=1}^M{\bf a}_j$.  Let $C(A)\subseteq{\mathbb R}^n$ be the real cone generated by $A$ and let $\sigma_\beta$ be the smallest closed face of $C(A)$ containing $\beta$.  Let $\sigma_\beta^\circ$ be the relative interior of $\sigma_\beta$, i.~e., $\sigma_\beta^\circ$ equals $\sigma_\beta$ minus all its proper closed subfaces.  Then $\beta\in\sigma_\beta^\circ$.  We make the assumption throughout this paper that $A'$ is {\it minimal for $\sigma_\beta^\circ$\/}, i.~e., that if $J$ is a proper subset of $\{1,\dots,M\}$, then $\sum_{j\in J} {\bf a}_j\not\in\sigma_\beta^{\circ}$.  

Let ${\mathbb Z}A\subseteq{\mathbb Z}^n$ be the abelian group generated by $A$, let ${\mathbb N}$ be the nonnegative integers, and let ${\mathbb N}A$ be the semigroup generated by $A$.  Put ${\mathcal M}_\beta = (-\sigma_\beta^\circ)\cap{\mathbb Z}A$.  
\begin{lemma}
Assume that $A'$ is minimal for $\sigma_\beta^\circ$, let $u\in {\mathcal M}_\beta$, and suppose that
$u = \sum_{j=1}^N l_j{\bf a}_j$ with $l_j\in{\mathbb Z}$ for $j=1,\dots,M$ and $l_j\in{\mathbb N}$ for $j=M+1,\dots,N$.  Then $l_j<0$ for $j=1,\dots,M$ and $l_j=0$  for all $j\in\{M+1,\dots,N\}$ for which ${\bf a}_j\not\in\sigma_\beta$.  
\end{lemma}

\begin{proof}
Let $H\subseteq{\mathbb R}^n$ be a hyperplane of support for the face $\sigma_\beta$.  Then $H$ is defined by a homogeneous linear equation $h=0$ with $h(\sigma_\beta) = 0$ and $h(x)>0$ for $x\in C(A)\setminus\sigma_\beta$.  If $l_j>0$ for some $j\in\{M+1,\dots,N\}$ with ${\bf a}_j\not\in\sigma_\beta$, then $h(u)>0$, contradicting the hypothesis that $u\in{\mathcal M}_\beta$.  We then have 
\[ \sum_{\substack{j\in\{1,\dots,M\}\\ l_j<0}} l_j{\bf a}_j = u-\sum_{\substack{j\in\{1,\dots,M\}\\ l_j>0}}l_j{\bf a}_j - \sum_{j=M+1}^N l_j{\bf a}_j\in {\mathcal M}_\beta, \]
hence $\sum_{\substack{j\in\{1,\dots,M\}\\ l_j<0}} {\bf a}_j\in{\mathcal M}_\beta$.  The assumption that $A'$ is minimal for $\sigma_\beta^\circ$ now implies that $l_j<0$ for $j=1,\dots,M$.
\end{proof}

Put 
\[ E_\beta = \{ l=(l_1,\dots,l_N)\in{\mathbb N}^N\mid \text{$l_j=0$ if ${\bf a}_j\not\in\sigma_\beta$}\}. \]
For $u\in{\mathcal M}_\beta$ define
\begin{multline}
 F_u(\Lambda_1,\dots,\Lambda_N) = \\ 
\sum_{\substack{l\in E_\beta\\ \sum_{j=1}^M (-l_j-1){\bf a}_j +\\ \sum_{j=M+1}^N l_j{\bf a}_j= u}} 
(-1)^{\sum_{j=1}^M l_j}\frac{\prod_{j=1}^M l_j!}{\prod_{j=M+1}^N l_j!} \Lambda_1^{-l_1-1}\cdots\Lambda_M^{-l_M-1}\Lambda_{M+1}^{l_{M+1}}\cdots\Lambda_N^{l_N}. 
\end{multline}
We show that the assumption that $A'$ be minimal for $\sigma_\beta^\circ$ implies that the $F_u(\Lambda)$ satisfy $A$-hypergeometric systems.

Let $L$ be the lattice of relations on the set $A$:
\[ L = \bigg\{ l=(l_1,\dots,l_N)\in{\mathbb Z}^N\mid \sum_{j=1}^N l_j{\bf a}_j = {\bf 0}\bigg\}. \]
For $l\in L$ define the box operator $\Box_l$ by
\begin{equation}
\Box_l = \prod_{l_j>0} \bigg(\frac{\partial}{\partial\Lambda_j}\bigg)^{l_j} - \prod_{l_j<0} \bigg(\frac{\partial}{\partial \Lambda_j}\bigg)^{-l_j}.
\end{equation}
The Euler operators for a parameter $u=(u_1,\dots,u_n)\in {\mathbb C}^n$ are defined by 
\begin{equation}
Z_i = \sum_{j=1}^N a_{ij}\Lambda_j\frac{\partial}{\partial\Lambda_j} - u_i
\end{equation}
for $i=1,\dots,n$.  The {\it $A$-hypergeometric system with parameter $u$\/} is the system of partial differential equations consisting of the box operators $\Box_l$ for $l\in L$ and the $Z_i$ for $i=1,\dots,n$.

\begin{lemma}
Suppose that $A'$ is minimal for $\sigma_\beta^\circ$ and let $u\in{\mathcal M}_\beta$.  Then
\[ \frac{\partial}{\partial\Lambda_k}F_u(\Lambda) = \begin{cases} F_{u-{\bf a}_k}(\Lambda) & \text{if ${\bf a}_k\in\sigma_\beta$,} \\ 0 & \text{if ${\bf a}_k\not\in\sigma_\beta$.} \end{cases} \]
\end{lemma}

\begin{proof}
If ${\bf a}_k\not\in\sigma_\beta$, then $\Lambda_k$ does not occur to a nonzero power in (2.2), so $\partial F_u/\partial\Lambda_k = 0$.  
A straightforward calculation from (2.2) shows that applying $\partial/\partial\Lambda_k$ to a term in $F_u(\Lambda)$ gives either 0 or a term of $F_{u-{\bf a}_k}(\Lambda)$.  The main point of the proof is to show that every monomial in $F_{u-{\bf a}_k}(\Lambda)$ is obtained by applying $\partial/\partial\Lambda_k$ to some monomial in $F_u(\Lambda)$.  

Suppose that ${\bf a}_k\in\sigma_\beta$ and consider a monomial in $F_{u-{\bf a}_k}(\Lambda)$:
\begin{equation}
(-1)^{\sum_{j=1}^M l_j}\frac{\prod_{j=1}^M l_j!}{\prod_{j=M+1}^N l_j!} \Lambda_1^{-l_1-1}\cdots\Lambda_M^{-l_M-1}\Lambda_{M+1}^{l_{M+1}}\cdots\Lambda_N^{l_N} 
\end{equation}
with $l\in E_\beta$ and 
\begin{equation}
\sum_{j=1}^M (-l_j-1){\bf a}_j +\sum_{j=M+1}^N l_j{\bf a}_j= u-{\bf a}_k.
\end{equation}  
Suppose first that $k\in\{1,\dots,M\}$.  Then (2.7) gives
\[ u = -l_k{\bf a}_k + \sum_{\substack{j=1\\ j\neq k}}^M (-l_j-1){\bf a}_j +\sum_{j=M+1}^N l_j{\bf a}_j. \]
By Lemma 2.1 we have $-l_k<0$, hence 
\[ (l_1,\dots,l_{k-1},l_{k}-1,l_{k+1},\dots,l_N)\in E_\beta, \]
the monomial
\[ (-1)^{-1-\sum_{j=1}^M l_j}\frac{(l_k-1)!\prod_{\substack{j=1\\ j\neq k}}^M l_j!}{\prod_{j=M+1}^N l_j!} \Lambda_k \cdot\big(\Lambda_1^{-l_1-1}\cdots\Lambda_M^{-l_M-1}\Lambda_{M+1}^{l_{M+1}}\cdots\Lambda_N^{l_N}\big) \]
appears in the series (2.2), and applying $\partial/\partial\Lambda_k$ gives (2.6).  Next suppose that $k\in\{M+1,\dots,N\}$.   Then the monomial 
\[ (-1)^{\sum_{j=1}^M l_j}\frac{\prod_{j=1}^M l_j!}{(l_k+1)!\prod_{\substack{j=M+1\\ j\neq k}}^N l_j!} 
 \Lambda_k \cdot\big(\Lambda_1^{-l_1-1}\cdots\Lambda_M^{-l_M-1}\Lambda_{M+1}^{l_{M+1}}\cdots\Lambda_N^{l_N}\big) \]
appears in the series (2.2) and applying $\partial/\partial\Lambda_k$ gives (2.6).
\end{proof}

\begin{corollary}
Suppose that $A'$ is minimal for $\sigma_\beta^\circ$ and that $u\in{\mathcal M}_\beta$.  Then $F_u(\Lambda)$ satisfies the $A$-hypergeometric system with parameter $u$.  
\end{corollary}

\begin{proof}
A straightforward calculation from (2.2) shows that for $u\in{\mathcal M}_\beta$ the series $F_u(\Lambda)$ satisfies the operators $Z_i$ for the parameter $u$.  Let $l\in L$.  Then
\begin{equation}
\sum_{l_j>0}l_j{\bf a}_j = -\sum_{l_j<0}l_j{\bf a}_j. 
\end{equation}
If $l_j>0$ for some ${\bf a}_j\not\in\sigma_\beta$, this equality implies that $l_{j'}<0$ for some ${\bf a}_{j'}\not\in\sigma_\beta$.  Then by Lemma 2.5
\[ \prod_{l_j>0} \bigg(\frac{\partial}{\partial\Lambda_j}\bigg)^{l_j}F_u(\Lambda) = 
 \prod_{l_j<0} \bigg(\frac{\partial}{\partial\Lambda_j}\bigg)^{-l_j}F_u(\Lambda) = 0, \]
so $\Box_lF_u(\Lambda)=0$.  Otherwise, Lemma 2.5 implies that
\[ \prod_{l_j>0} \bigg(\frac{\partial}{\partial\Lambda_j}\bigg)^{l_j}F_u(\Lambda) = F_{u-\sum_{l_j>0} l_j{\bf a}_j}(\Lambda) \]
and
\[  \prod_{l_j<0} \bigg(\frac{\partial}{\partial\Lambda_j}\bigg)^{-l_j}F_u(\Lambda) = F_{u+\sum_{l_j<0} l_j{\bf a}_j}(\Lambda). \]
By Equation~(2.9) these two expressions are equal, so again $\Box_lF_u(\Lambda)=0$.  
\end{proof}

In earlier work (\cite{AS1,AS3}) we considered the special case $F_{-\beta}(\Lambda)$ and gave a criterion for that series to have integral coefficients.  However, for certain arithmetic applications it is necessary to know that all the series (2.2) have integral coefficients.  For example, in \cite{AS2} we proved this in a special case by ad hoc methods and used the result to give a $p$-adic formula for the unit root of the zeta function of generalized Calabi-Yau hypersurfaces in characteristic $p$.  The point of this article is to give a general criterion that guarantees all the series (2.2) have integral coefficients.

\begin{theorem}
Assume that $A'$ is minimal for $\sigma_\beta^\circ$.  If 
\begin{equation}
M = \min\{-u_n\mid u=(u_1,\dots,u_n)\in{\mathcal M}_\beta\text{ and }F_u(\Lambda)\neq 0\},
\end{equation}
then for all $u\in{\mathcal M}_\beta$ the series $F_u(\Lambda)$ has integral coefficients.
\end{theorem}

Note that the hypothesis of the theorem is redundant: Eq.~(2.11) implies that $A'$ is minimal for $\sigma_\beta^\circ$.  
Note also that by Equation (2.2), the condition that $F_u(\Lambda) = 0$ is equivalent to the condition that the equation
\[ \sum_{j=1}^M (-l_j-1){\bf a}_j + \sum_{j=M+1}^N l_j{\bf a}_j =u \]
have no solution $l\in E_\beta$.  

To fix ideas and simplify the notation, we shall assume from now on that $\sigma_\beta = C(A)$, i.~e., that $\beta$ is an interior point of $C(A)$.  The case of general $\beta$ is proved along the same lines.  In this case we write ${\mathcal M}$ in place of ${\mathcal M}_\beta$ and we have ${\mathcal M} = (-C(A)^\circ)\cap{\mathbb Z}A$.  We also have $E_\beta = {\mathbb N}^N$.   With this assumption, Lemma~2.1 becomes the following statement.
\begin{lemma}
Assume that $A'$ is minimal for $C(A)^\circ$.  Let $u\in {\mathcal M}$ and suppose that
$u = \sum_{j=1}^N l_j{\bf a}_j$ with $l_j\in{\mathbb Z}$ for $j=1,\dots,M$ and $l_j\in{\mathbb N}$ for $j=M+1,\dots,N$.  Then $l_j<0$ for $j=1,\dots,M$.
\end{lemma}

Theorem 2.10 becomes the following statement.
\begin{theorem}
Assume that $\sum_{j=1}^M {\bf a}_j$ is an interior point of $C(A)$.  If 
\[ M = \min\{-u_n\mid u=(u_1,\dots,u_n)\in {\mathcal M}\text{ and }F_u(\Lambda)\neq 0\}, \]
then for all $u\in{\mathcal M}$ the series $F_u(\Lambda)$ has integral coefficients.
\end{theorem}

{\bf Remark.}  The restriction that the last coordinate of each ${\bf a}_j$ equal 1 is not essential.  Theorem~2.10 remains true if we assume only that all ${\bf a}_j$ lie on a hyperplane $\sum_{i=1}^n c_i u_i=1$ with $c_i\in{\mathbb Z}$ for $i=1,\dots,n$.  When this condition is satisfied, there exists a linear transformation $T:{\mathbb R}^n\to{\mathbb R}^n$ which restricts to an isomorphism of abelian groups $T:{\mathbb Z}^n\to{\mathbb Z}^n$ and which transforms the hyperplane $\sum_{i=1}^n c_iu_i=1$ to the hyperplane $u_n=1$.  The set $T(A) = \{T({\bf a}_j)\}_{j=1}^N$ then satisfies the hypothesis of Theorem~2.10, $l\in E_\beta$ if and only if $l\in E_{T(\beta)}$, and
\[ \sum_{j=1}^M (-l_j-1){\bf a}_j+\sum_{j=M+1}^N l_j{\bf a}_j = u \]
if and only if
\[ \sum_{j=1}^M (-l_j-1)T({\bf a}_j)+\sum_{j=M+1}^N l_jT({\bf a}_j) = T(u). \]
The series $F_u(\Lambda)$ and $F_{T(u)}(\Lambda)$ are thus identical.

\section{Generating series}

It will be convenient to have a generating series construction of the~$F_u(\Lambda)$.
Consider the series
\[ \zeta(t) = \sum_{i=0}^\infty (-1)^ii!t^{-i-1}. \]
A straightforward calculation shows that for $u\in{\mathcal M}$, $F_u(\Lambda)$ is the coefficient of $x^u$~in 
\[ \bigg(\prod_{j=1}^M \zeta(\Lambda_jx^{{\bf a}_j})\bigg) \bigg(\prod_{j=M+1}^N \exp(\Lambda_jx^{{\bf a}_j})\bigg). \]
We define
\[ \delta_{\mathcal M}(x^u) = \begin{cases} x^u & \text{if $u\in{\mathcal M}$,} \\ 0 & \text{if $u\not\in{\mathcal M}$.} \end{cases} \]
Then
\begin{equation}
\delta_{\mathcal M}\bigg( \bigg(\prod_{j=1}^M \zeta(\Lambda_jx^{{\bf a}_j})\bigg) \bigg(\prod_{j=M+1}^N \exp(\Lambda_jx^{{\bf a}_j})\bigg)\bigg) = \sum_{u\in{\mathcal M}} F_u(\Lambda)x^u.
\end{equation}

To prove Theorem 2.13, we fix a prime $p$ and prove that all $F_u(\Lambda)$ have $p$-integral coefficients.  We begin by normalizing the series (3.1) for the prime $p$.  

We recall some estimates from \cite[Section 3]{AS2} for this purpose.
Let $ {\rm AH}(t)= \exp\big(\sum_{i=0}^{\infty}t^{p^i}/p^i\big)$ be the Artin-Hasse series, a power series in $t$ that has $p$-integral coefficients, let $\gamma_0$ be a zero of the series $\sum_{i=0}^\infty t^{p^i}/p^i$ having ${\rm ord}\;\gamma_0 = 1/(p-1)$, and set 
\[ \theta(t) = {\rm AH}({\gamma}_0t)=\sum_{i=0}^{\infty}\theta_i t^i. \]
We then have
\begin{equation}
{\rm ord}\: \theta_i\geq \frac{i}{p-1}.
\end{equation}

We define $\hat{\theta}(t) = \prod_{j=0}^\infty \theta(t^{p^j})$, which gives $\theta(t) = \hat{\theta}(t)/\hat{\theta}(t^p)$.   If we set 
\begin{equation}
\gamma_j = \sum_{i=0}^j \frac{\gamma_0^{p^i}}{p^i},
\end{equation}
then
\begin{equation}
\hat{\theta}(t) = \exp\bigg(\sum_{j=0}^{\infty} \gamma_j t^{p^j}\bigg) = \prod_{j=0}^\infty \exp(\gamma_j t^{p^j}).
\end{equation}
If we write $\hat{\theta}(t) = \sum_{i=0}^\infty \hat{\theta}_i(\gamma_0 t)^i/i!$, then by \cite[Equation~(3.8)]{AS2} we have
\begin{equation}
{\rm ord}\:\hat{\theta}_i\geq 0.
\end{equation}

We also need the series
\begin{equation}
\hat{\theta}_1(t) := \prod_{j=1}^\infty \exp(\gamma_jt^{p^j}) = :\sum_{i=0}^\infty \frac{\hat{\theta}_{1,i}}{i!}(\gamma_0 t)^i. 
\end{equation}
Note that $\hat{\theta}(t) = \exp(\gamma_0t)\hat{\theta}_1(t)$.  
By \cite[Equation~(3.10)]{AS2}
\begin{equation}
{\rm ord}\:\hat{\theta}_{1,i}\geq \frac{i(p-1)}{p}. 
\end{equation}

Define the generating series $F(\Lambda,x)$ by the formula
\begin{equation}
F(\Lambda,x) =\delta_{\mathcal M}\bigg( \bigg(\prod_{j=1}^M \zeta(\gamma_0\Lambda_jx^{{\bf a}_j})\bigg) \bigg(\prod_{j=M+1}^N \exp(\gamma_0\Lambda_jx^{{\bf a}_j})\bigg)\bigg).
\end{equation}
It follows from (3.1) that
\begin{equation}
F(\Lambda,x) = \sum_{u\in{\mathcal M}} F_u(\Lambda)\gamma_0^{u_n}x^u.
\end{equation}

Our $p$-adically normalized series $G(\Lambda,x)$ is defined by
\begin{multline}
G(\Lambda,x) = \\
\delta_{\mathcal M}\bigg( \bigg(\prod_{j=1}^M \zeta(\gamma_0\Lambda_jx^{{\bf a}_j})\hat{\theta}_1(\Lambda_jx^{{\bf a}_j})\bigg) \bigg(\prod_{j=M+1}^N \exp(\gamma_0\Lambda_jx^{{\bf a}_j})\hat{\theta}_1(\Lambda_jx^{{\bf a}_j})\bigg)\bigg),
\end{multline}
which we write as
\begin{equation}
G(\Lambda,x) = \sum_{u\in{\mathcal M}} G_u(\Lambda)\gamma_0^{u_n}x^u.
\end{equation}

The first step is to evaluate (3.10) to obtain an explicit formula for the $G_u(\Lambda)$.  We introduce some notation to simplify formulas.  Let $Q(t)$ be the sum of the negative powers of $t$ in the product $\zeta(\gamma_0t)\hat{\theta}_1(t)$.  A straightforward calculation shows that
\begin{equation}
Q(t) = \sum_{i=0}^\infty (-1)^i i! \sigma_i \gamma_0^{-i-1}t^{-i-1},
\end{equation}
where
\begin{equation}
\sigma_i = \sum_{k=0}^\infty (-1)^k \hat{\theta}_{1,k} \binom{i+k}{k}.
\end{equation}
Since $\hat{\theta}_{1,0} = 1$, Equation~(3.7) shows that $\sigma_i$ is $p$-integral and that
\begin{equation}
\sigma_i \equiv 1\pmod{\gamma_0}.
\end{equation}
It follows from Lemma 2.12 that to obtain an $x^u$ with $u\in{\mathcal M}$, one must take negative powers of each $x^{{\bf a}_j}$, $j=1,\dots,M$, in the first product on the right-hand side of (3.10).  Thus
\begin{equation}
G(\Lambda,x) = 
\delta_{\mathcal M}\bigg( \bigg(\prod_{j=1}^M Q(\Lambda_jx^{{\bf a}_j})\bigg) \bigg(\prod_{j=M+1}^N \exp(\gamma_0\Lambda_jx^{{\bf a}_j})\hat{\theta}_1(\Lambda_jx^{{\bf a}_j})\bigg)\bigg),
\end{equation}

A similar calculation shows that
\begin{equation}
\exp(\gamma_0\Lambda_jx^{{\bf a}_j})\hat{\theta}_1(\Lambda_jx^{{\bf a}_j}) = \sum_{i=0}^\infty \frac{\tau_i\gamma_0^i\Lambda_j^ix^{i{\bf a}_j}}{i!},
\end{equation}
where
\begin{equation}
\tau_i = \sum_{k=0}^i \binom{i}{k}\hat{\theta}_{1,k}.
\end{equation}
Equation~(3.7) shows that $\tau_i$ is $p$-integral and that
\begin{equation}
\tau_i \equiv 1\pmod{\gamma_0}.
\end{equation}
Set for $l\in{\mathbb N}^N$
\[ \sigma(l) = \prod_{j=1}^M \sigma_{l_j}\prod_{j=M+1}^N \tau_{l_j}. \]
By (3.14) and (3.18), $\sigma(l)$ is $p$-integral and
\begin{equation}
\sigma(l)\equiv 1 \pmod{\gamma_0}.
\end{equation}

It now follows from (3.12), (3.15), and (3.16) that for $u\in{\mathcal M}$
\begin{multline}
G_u(\Lambda) = \\
\sum_{\substack{l\in{\mathbb N}^N\\ \sum_{j=1}^M (-l_j-1){\bf a}_j +\\ \sum_{j=M+1}^N l_j{\bf a}_j= u}} 
(-1)^{\sum_{j=1}^M l_j}\frac{\sigma(l)\prod_{j=1}^M l_j!}{\prod_{j=M+1}^N l_j!} \Lambda_1^{-l_1-1}\cdots\Lambda_M^{-l_M-1}\Lambda_{M+1}^{l_{M+1}}\cdots\Lambda_N^{l_N}. 
\end{multline}
Equations (2.2) and (3.20), together with (3.19), now imply the following proposition.
\begin{proposition}
If $A'$ is minimal for $C(A)^\circ$, then for $u\in{\mathcal M}$ the series $F_u(\Lambda)$ has $p$-integral coefficients if and only if the series $G_u(\Lambda)$ has $p$-integral coefficients.
\end{proposition}

Note that since $\exp (\gamma_0t) \hat{\theta}_1(t)=\hat{\theta}(t)$, Equation (3.15) implies that $G(\Lambda,x)$ can be written in the more compact form
\begin{equation}
G(\Lambda,x) = \delta_{\mathcal M}\bigg(\bigg(\prod_{j=1}^M Q(\Lambda_jx^{{\bf a}_j})\bigg)\bigg(\prod_{j=M+1}^N \hat{\theta}(\Lambda_jx^{{\bf a}_j})\bigg)\bigg),
\end{equation}
and Equation (3.16) implies that
\begin{equation}
\hat{\theta}(\Lambda_jx^{{\bf a}_j}) = \sum_{i=0}^\infty \frac{\tau_i\gamma_0^i\Lambda_j^ix^{i{\bf a}_j}}{i!}.
\end{equation}

\section{The Dwork-Frobenius operator}

By Proposition 3.21, it suffices to prove that the series $G_u(\Lambda)$ have $p$-integral coefficients.  The main step in the proof will be to show that $G(\Lambda,x)$ is an eigenvector of the Dwork-Frobenius operator.

Define
\begin{equation}
\theta(\Lambda,x) = \prod_{j=1}^N \theta(\Lambda_jx^{{\bf a}_j}).
\end{equation}
Let ${\mathbb N}A$ be the semigroup generated by $A$.  By \cite[Equations~(3.20)--(3.22)]{AS2} we have
\begin{equation}
\theta(\Lambda,x) = \sum_{\nu\in{\mathbb N}A} \theta_\nu(\Lambda)x^\nu,
\end{equation}
where $\theta_\nu(\Lambda)$ is a polynomial in $\Lambda$:
\begin{equation}
\theta_\nu(\Lambda) = \sum_{\substack{m\in{\mathbb N}^N\\ \sum_{j=1}^N m_j{\bf a}_j = \nu}} \theta_m^{(\nu)}\Lambda^m
\end{equation}
with
\begin{equation}
\theta_m^{(\nu)} = \prod_{j=1}^N \theta_{m_j}.
\end{equation}
We have by (3.2)
\begin{equation}
\text{ord}\: \theta_m^{(\nu)}\geq \frac{\sum_{j=1}^N m_j}{p-1} = \frac{\nu_n}{p-1}.
\end{equation}

The Dwork-Frobenius operator $\alpha^*$ on $G(\Lambda,x)$ is defined by
\begin{equation}
\alpha^*\big(G(\Lambda,x)\big) = \delta_{\mathcal M}\bigg( \theta(\Lambda,x)G(\Lambda^p,x^p)\bigg).
\end{equation}
We check that this operation is well-defined.  Formally, 
\[ \alpha^*\big(G(\Lambda,x)\big) = \sum_{\rho\in{\mathcal M}} \zeta_\rho(\Lambda)x^\rho, \]
where
\begin{equation}
\zeta_\rho(\Lambda) = \sum_{\substack{\nu\in{\mathbb N}A,\, u\in{\mathcal M}\\ \nu+pu=\rho}} \gamma_0^{u_n} \theta_\nu(\Lambda) G_u(\Lambda^p).
\end{equation}
The product $\theta_\nu(\Lambda)G_u(\Lambda^p)$ is well-defined because $\theta_\nu(\Lambda)$ is a polynomial.  We show that the series on the right-hand side of (4.7) converges by computing the coefficient of $\Lambda^\mu$ in this expression.

From (3.20), (4.3), and (4.4) the coefficient of $\Lambda^\mu$ in $\gamma_0^{u_n}\theta_\nu(\Lambda)G_u(\Lambda^p)$ is 
\begin{equation}
\sum (-1)^{\sum_{j=1}^M l_j} \sigma(l) \bigg(\prod_{j=1}^M \gamma_0^{-l_j-1}\theta_{m_j}l_j!\bigg)\bigg(\prod_{j=M+1}^N \frac{\gamma_0^{l_j}\theta_{m_j}}{l_j!}\bigg),
\end{equation}
where the sum is over the (finite) set of all $m\in{\mathbb N}^N$ and $l\in{\mathbb N}^N$ satisfying
\begin{equation}
\sum_{j=1}^N m_j{\bf a}_j = \nu,
\end{equation}
\begin{equation}
\sum_{j=1}^M (-l_j-1){\bf a}_j+\sum_{j=M+1}^N l_j{\bf a}_j = u,
\end{equation}
and 
\begin{equation}
\mu_j = \begin{cases} m_j+p(-l_j-1) & \text{for $j=1,\dots,M$,}\\ m_j+pl_j & \text{for $j=M+1,\dots,N$,} \end{cases}
\end{equation}
and where we have written $\gamma_0^{u_n} = \prod_{j=1}^M \gamma_0^{-l_j-1}\cdot\prod_{j=M+1}^N\gamma_0^{l_j}$, which follows from (4.10) since the last coordinate of each ${\bf a}_j$ equals 1.

For the last factor on the right-hand side of (4.8) we have
\[ \text{ord}\:\prod_{j=M+1}^N \frac{\gamma_0^{l_j}\theta_{m_j}}{l_j!}\geq \sum_{j=M+1}^N \frac{m_j+s_{l_j}}{p-1}, \]
where for a nonnegative integer $a$, $s_a$ denotes the sum of the digits in the $p$-adic expansion of $a$.  In particular, this factor is $p$-integral.  For the next-to-last factor on the right-hand side of (4.8) we have
\[ \text{ord}\:\prod_{j=1}^M \gamma_0^{-l_j-1}\theta_{m_j}l_j!\geq \sum_{j=1}^M \frac{m_j-1-s_{l_j}}{p-1}. \]
We have the elementary estimate $s_{l_j}/(p-1)\leq\log_p(pl_j)$ and from (4.11) we get $m_j = p(l_j+1)+\mu_j$, which gives the estimate
\begin{equation} \frac{m_j-1-s_{l_j}}{p-1}\geq \frac{p(l_j+1)+\mu_j-1}{p-1}-\log_p(pl_j). \end{equation}
As $u$ grows without bound on the right-hand side of (4.7), the sum $\sum_{j=1}^M l_j$ grows without bound also, so estimate (4.12) implies that the series (4.7) converges.

We normalize the series (4.7) to be consistent with the normalization of $G(\Lambda,x)$.  Write
\begin{equation}
\alpha^*\big(G(\Lambda,x)\big) = \sum_{\rho\in{\mathcal M}} \eta_\rho(\Lambda)\gamma_0^{\rho_n}x^\rho,
\end{equation}
where
\begin{equation}
\eta_\rho(\Lambda) = \sum_{\substack{\nu\in{\mathbb N}A,\, u\in{\mathcal M}\\ \nu+pu=\rho}} \gamma_0^{-\rho_n +u_n} \theta_\nu(\Lambda) G_u(\Lambda^p).
\end{equation}

For later use, we record the formula for $\zeta_\rho(\Lambda)$.  This follows from Equations (4.7)--(4.11).
\begin{proposition}
For $\rho\in{\mathcal M}$, the coefficient of $x^\rho$ in $\alpha^*\big(G(\Lambda,x)\big)$ equals
\begin{multline}
\sum_{\substack{m,l\in{\mathbb N}^N\\ \sum_{j=1}^M (m_j+p(-l_j-1)){\bf a}_j + \\
\sum_{j=M+1}^N (m_j+pl_j){\bf a}_j = \rho}}  (-1)^{\sum_{j=1}^M l_j} \sigma(l) \bigg(\prod_{j=1}^M \gamma_0^{-l_j-1}\theta_{m_j}l_j!\bigg)\bigg(\prod_{j=M+1}^N \frac{\gamma_0^{l_j}\theta_{m_j}}{l_j!}\bigg)\\
\cdot\prod_{j=1}^M \Lambda_j^{m_j+p(-l_j-1)} \prod_{j=M+1}^N \Lambda_j^{m_j+pl_j}.
\end{multline}
The coefficient of $\Lambda^\mu$ in this expression is obtained by restricting $m,l$ to satisfy~(4.11).
\end{proposition}

\section{Eigenvector of Dwork-Frobenius}

The following result is key to the proof of Theorem 2.13.
\begin{theorem}
If $A'$ is minimal for $C(A)^\circ$, then the series $G(\Lambda,x)$ is an eigenvector of $\alpha^*$ with eigenvalue $p^M$:
\[ \alpha^*\big(G(\Lambda,x)\big) = p^M G(\Lambda,x). \]
\end{theorem}

We begin by recalling a result from \cite{AS2}.  Define an operator $\delta_-$ on formal Laurent series in one variable $t$ by the formula
\[ \delta_-\bigg(\sum_{i=-\infty}^\infty c_it^i\bigg) = \sum_{i=-\infty}^{-1} c_it^i. \]
The following equality is \cite[Proposition 6.10]{AS2}:
\begin{equation}
\delta_-\big(\theta(t)Q(t^p)\big) = pQ(t).
\end{equation}
This equation is equivalent to the assertion that
\begin{equation}
\theta(t)Q(t^p)= A(t) + pQ(t)
\end{equation}
for some series $A(t)$ in nonnegative powers of $t$.  Replacing $t$ in this equation by $\Lambda_jx^{{\bf a}_j}$ for $j=1,\dots,M$ and multiplying gives
\begin{equation}
\prod_{j=1}^M \bigg(\theta(\Lambda_jx^{{\bf a}_j})Q(\Lambda_j^p x^{p{\bf a}_j})\bigg) = \prod_{j=1}^M \bigg( A(\Lambda_jx^{{\bf a}_j}) + pQ(\Lambda_j x^{{\bf a}_j})\bigg),
\end{equation}
where $A(\Lambda_jx^{{\bf a}_j})$ is a series in nonnegative powers of $x^{{\bf a}_j}$.  

\begin{lemma}
If $A'$ is minimal for $C(A)^\circ$, then 
\begin{equation}
\delta_{\mathcal M}\prod_{j=1}^M \bigg(\theta(\Lambda_jx^{{\bf a}_j})Q(\Lambda_j^p x^{p{\bf a}_j})\bigg) = p^M\prod_{j=1}^M Q(\Lambda_jx^{{\bf a}_j}).
\end{equation}
\end{lemma}

\begin{proof}
It follows from Lemma 2.12 that when the product on the right-hand side of~(5.4) is expanded, all terms except for $\prod_{j=1}^M pQ(\Lambda_jx^{{\bf a}_j})$ are annihilated by $\delta_{\mathcal M}$.
\end{proof}

Since $\theta(t)\hat{\theta}(t^p) = \hat{\theta}(t)$, we have trivially
\begin{equation}
\prod_{j=M+1}^N \theta(\Lambda_jx^{{\bf a}_j})\hat{\theta}(\Lambda_j^px^{p{\bf a}_j}) = \prod_{j=M+1}^N \hat{\theta}(\Lambda_jx^{{\bf a}_j}).
\end{equation}
\begin{corollary}
If $A'$ is minimal for $C(A)^\circ$, then
\begin{multline}
\delta_{\mathcal M} \bigg(\bigg(\prod_{j=1}^M \theta(\Lambda_jx^{{\bf a}_j})Q(\Lambda_j^p x^{p{\bf a}_j})\bigg)\bigg(\prod_{j=M+1}^N \theta(\Lambda_jx^{{\bf a}_j})\hat{\theta}(\Lambda_j^px^{p{\bf a}_j}) \bigg)\bigg) \\
= p^MG(\Lambda,x).
\end{multline}
\end{corollary}

\begin{proof}
It follows from Lemma 2.12 that the left-hand side of (5.9) equals
\[ \delta_{\mathcal M} \bigg(\bigg(\delta_{\mathcal M}\prod_{j=1}^M \theta(\Lambda_jx^{{\bf a}_j})Q(\Lambda_j^p x^{p{\bf a}_j})\bigg)\bigg(\prod_{j=M+1}^N \theta(\Lambda_jx^{{\bf a}_j})\hat{\theta}(\Lambda_j^px^{p{\bf a}_j}) \bigg)\bigg) \]
By (5.6) and (5.7) this equals
\[ \delta_{\mathcal M}\bigg(p^M\prod_{j=1}^M Q(\Lambda_jx^{{\bf a}_j})\prod_{j=M+1}^N \hat{\theta}(\Lambda_jx^{{\bf a}_j})\bigg). \]
The assertion of the corollary now follows from (3.22)
\end{proof}

\begin{proof}[Proof of Theorem 5.1]
From (3.22) and (4.6) we have
\begin{multline*}
\alpha^*\big(G(\Lambda,x)\big) = \\ 
\delta_{\mathcal M} \bigg( \bigg(\prod_{j=1}^N \theta(\Lambda_jx^{{\bf a}_j})\bigg) \bigg(\delta_{\mathcal M}\bigg(\prod_{j=1}^M Q(\Lambda_j^p x^{p{\bf a}_j})\prod_{j=M+1}^N \hat{\theta}(\Lambda_j^p x^{p{\bf a}_j})\bigg)\bigg) \bigg).
\end{multline*}
By Lemma 2.12 we may omit the second ``$\delta_{\mathcal M}$'' to get
\begin{multline}
\alpha^*\big(G(\Lambda,x)\big) = \\ 
\delta_{\mathcal M} \bigg( \bigg(\prod_{j=1}^N \theta(\Lambda_jx^{{\bf a}_j})\bigg) \bigg(\bigg(\prod_{j=1}^M Q(\Lambda_j^p x^{p{\bf a}_j})\prod_{j=M+1}^N \hat{\theta}(\Lambda_j^p x^{p{\bf a}_j})\bigg) \bigg)\bigg).
\end{multline}
To prove Theorem 5.1, it thus suffices to show that the left-hand side of (5.9) equals the right-hand side of (5.10).
Formally, the products involved are just rearrangements of each other.  However, we do not know a commutative ring that contains all these factors, so we prove equality by computing each separately and verifying that the results are the same.
For the right-hand side of (5.10), the result of this calculation was given in Proposition~4.15, so it remains only to perform this calculation for the left-hand side of (5.9).

Consider first a product $\theta(\Lambda_jx^{{\bf a}_j})Q(\Lambda_j^px^{p{\bf a}_j})$, $j=1,\dots,M$.  By Lem\-ma~2.12 we can ignore all terms in this product with $x^{{\bf a}_j}$ raised to a nonnegative power, as all monomials to which they contribute will be annihilated by $\delta_{\mathcal M}$.  From the definitions, for $\rho'\in{\mathcal M}$, the coefficient of $x^{\rho'}$ in $\prod_{j=1}^M \theta(\Lambda_jx^{{\bf a}_j})Q(\Lambda_j^px^{p{\bf a}_j})$ is
\begin{equation}
\sum_{\substack{m,l\in{\mathbb N}^M\\ \sum_{j=1}^M (m_j+p(-l_j-1)){\bf a}_j = \rho'}}\prod_{j=1}^M \theta_{m_j}(-1)^{l_j}l_j!\sigma_{l_j}\gamma_0^{-l_j-1}\prod_{j=1}^M \Lambda_j^{m_j+p(-l_j-1)}.
\end{equation}
For $\rho''\in{\mathbb N}A$, the coefficient of $x^{\rho''}$ in $\theta(\Lambda_jx^{{\bf a}_j})\hat{\theta}(\Lambda_j^px^{p{\bf a}_j})$ is (using (3.23))
\begin{equation}
\sum_{\substack{m,l\in{\mathbb N}^{N-M}\\ \sum_{j=M+1}^N (m_j+pl_j){\bf a}_j = \rho''}} \prod_{j=M+1}^N \frac{\theta_{m_j}\tau_{l_j}\gamma_0^{l_j}}{l_j!} \prod_{j=M+1}^N \Lambda_j^{m_j+pl_j}.
\end{equation}
Therefore, for $\rho\in {\mathcal M}$, the coefficient of $x^{\rho}$ on the left-hand side of (5.9) is
\begin{multline}
\sum_{\substack{m,l\in{\mathbb M}^N\\ \sum_{j=1}^M (m_j+p(-l_j-1)){\bf a}_j + \\ \sum_{j=M+1}^N (m_j+pl_j){\bf a}_j = \rho}} (-1)^{\sum_{j=1}^M l_j}\sigma(l) \prod_{j=1}^M \theta_{m_j}l_j!\gamma_0^{-l_j-1}\prod_{j=M+1}^N \frac{\theta_{m_j}\gamma_0^{l_j}}{l_j!}\\
\cdot\prod_{j=1}^M \Lambda_j^{m_j+p(-l_j-1)}\prod_{j=M+1}^N \Lambda_j^{m_j+pl_j}.
\end{multline}

This expression equals (4.16), which proves Theorem 5.1.
\end{proof}

\section{Proof of Theorem 2.13}

From Theorem 5.1 and Equation (4.14) we get the equality 
\begin{equation}
p^MG_\rho(\Lambda) = \sum_{\substack{\nu\in{\mathbb N}A,\,u\in{\mathcal M}\\ G_u(\Lambda)\neq 0,\, \nu+pu=\rho}} \gamma_0^{-\rho_n+u_n}\theta_\nu(\Lambda)G_u(\Lambda^p).
\end{equation}
From (4.3) and (4.5) every coefficient of $\theta_\nu(\Lambda)$ is divisible by $\gamma_0^{\nu_n}$.  If we put $\tilde{\theta}_\nu(\Lambda) = \gamma_0^{-\nu_n}\theta_\nu(\Lambda)$, a polynomial with $p$-integral coefficients, then we can rewrite this as
\begin{equation}
p^MG_\rho(\Lambda) = \sum_{\substack{\nu\in{\mathbb N}A,\,u\in{\mathcal M}\\ G_u(\Lambda)\neq 0,\,\nu+pu=\rho}} \gamma_0^{-\rho_n+u_n+\nu_n}\tilde{\theta}_\nu(\Lambda)G_u(\Lambda^p).
\end{equation}
By (2.2) and (3.20) we have $F_u(\Lambda)=0$ if and only if $G_u(\Lambda)=0$, so the hypothesis of Theorem 2.13 is equivalent to the condition that
\[ M=\min\{-u_n\mid u=(u_1,\dots,u_n)\in{\mathcal M}\text{ and }G_u(\Lambda)\neq 0\}. \]
This implies that $u_n\leq -M$ for $u\in{\mathcal M}$ such that $G_u(\Lambda)\neq 0$.  
The condition $\nu+pu=\rho$ implies that for such $u$
\begin{equation}
 -\rho_n+u_n+\nu_n = -(p-1)u_n= (p-1)M + (p-1)u_n', 
\end{equation}
where $u_n':=-u_n-M\in{\mathbb N}$.  Define $\tilde{\gamma} = \gamma_0^{(p-1)M}/p^M$, a $p$-adic unit.  From (6.2) and~(6.3) we get
\begin{equation}
G_\rho(\Lambda) = \sum_{\substack{\nu\in{\mathbb N}A,\,u\in{\mathcal M}\\ G_u(\Lambda)\neq 0,\, \nu+pu=\rho}} \tilde{\gamma}\gamma_0^{(p-1)u_n'}\tilde{\theta}_\nu(\Lambda)G_u(\Lambda^p),
\end{equation}
and all coefficients of the polynomial $\tilde{\gamma}\gamma_0^{(p-1)u'_n}\tilde{\theta}_\nu(\Lambda)$ are $p$-integral.

Consider a monomial $\Lambda^l$ in $G_u(\Lambda)$.  If $\sum_{j=M+1}^N l_j = 0$, then $l_j=0$ for $j=M+1,\dots,N$, so by (3.20) the coefficient of $\Lambda^l$ in $G_u(\Lambda)$ is $p$-integral for all $u\in{\mathcal M}$.  We proceed by induction on $d:=\sum_{j=M+1}^N l_j$.  Let $\rho\in{\mathcal M}$ and let $\Lambda^l$ be a monomial in $G_\rho(\Lambda)$ with $\sum_{j=M+1}^N l_j = d>0$ and suppose that all monomials $\Lambda^{l'}$ in all $G_u(\Lambda)$, $u\in{\mathcal M}$, with $\sum_{j=M+1}^N l'_j < d$ have $p$-integral coefficients.  The right-hand side of Equation (6.4) gives a formula for the coefficient of $\Lambda^l$ in $G_\rho(\Lambda)$.  But only monomials $\Lambda^{l'}$ of $G_u(\Lambda)$ with $\sum_{j=M+1}^N l'_j\leq d/p$ can contribute to this formula.  All these monomials have $p$-integral coefficients by the induction hypothesis, so (6.4) implies that the coefficient of $\Lambda^l$ in $G_\rho(\Lambda)$ is $p$-integral.

\section{Examples}

{\bf Example 1.}
Let $\{x^{{\bf b}_j}\}_{j=1}^N$ be the set of all monomials of degree $d$ in variables $x_0,\dots,x_n$ (so $N=\binom{d+n}{n}$) and let
\begin{equation}
f(x) = \sum_{j=1}^N \Lambda_jx^{{\bf b}_j} \in{\mathbb C}[\Lambda_1,\dots,\Lambda_N][x_0,\dots,x_n],
\end{equation}
the generic homogeneous polynomial of degree $d$ in $x_0,\dots,x_n$.  Put 
\[ {\bf a}_j = ({\bf b}_j,1)\in{\mathbb Z}^{n+2} \]
and let $A = \{{\bf a}_j\}_{j=1}^N$.  We use $u_0,\dots,u_{n+1}$ as the coordinate functions on ${\mathbb R}^{n+2}$.  The cone $C(A)$ lies in the hyperplane $\sum_{i=0}^n u_i = du_{n+1}$ and
${\mathbb Z}A$ consists of all lattice points on this hyperplane.  Thus
\begin{equation}
C(A)^\circ\cap{\mathbb Z}A = \bigg\{(u_0,\dots,u_{n+1})\mid \text{$\sum_{i=0}^n u_i=du_{n+1}$ and $u_i>0$ for all $i$}\bigg\}.
\end{equation}
In particular, 
\begin{equation}
\bigg\lceil \frac{n+1}{d}\bigg\rceil = \min\{ u_{n+1} \mid u=(u_0,\dots,u_{n+1})\in C(A)^\circ\cap{\mathbb Z}A\},
\end{equation}
where $\lceil r\rceil$ denotes the least integer greater than or equal to the real number $r$.

Put $M=\lceil \frac{n+1}{d}\rceil$.  There are many ways to choose a set $A = \{{\bf a}_j\}_{j=1}^M$ that satisfies the hypothesis of Theorem 2.13.  We give one example to illustrate the possibilities.  For $j=1,\dots,M-1$ take
\[ {\bf a}_j = (0,\dots,0,1,\dots,1,0,\dots,0,1), \]
where the middle group of $d$ ones occurs in positions $d(j-1),\dots,jd-1$, and take ${\bf a}_M$ to be any element of $A$ having nonzero entries in positions $(M-1)d,\dots,n$.  Then $\beta = \sum_{j=1}^M {\bf a}_j\in C(A)^\circ$ by (7.2), and by (7.3) the hypothesis of Theorem~2.13 is satisfied.  All the series (2.2) then have integral coefficients.  One can show that each $F_u(\Lambda)$, $u\in{\mathcal M}$, is a solution of the Picard-Fuchs equation of the hypersurface $f=0$ relative to an appropriate basis for $H^{n-1}_{\rm DR}(X/{\mathbb C}(\Lambda))$ (the basis depends on the choice of $u$).

We give an explicit example of this type where we set some of the $\Lambda_j$ equal to~0 to simplify formulas.  Consider the family of cubic surfaces in ${\mathbb P}^3$ defined by the polynomial
\[ f(x_0,\dots,x_3) = \Lambda_1 x_0x_1x_2 + \Lambda_2 x_1x_2x_3 + \sum_{i=3}^6 \Lambda_ix_{i-3}^3. \]
Then $A=\{{\bf a}_j\}_{j=1}^6\subseteq{\mathbb Z}^5$, where ${\bf a}_1 = (1,1,1,0,1)$, ${\bf a}_2 = (0,1,1,1,1)$, ${\bf a}_3 = (3,0,0,0,1)$, ${\bf a}_4 = (0,3,0,0,1)$, ${\bf a}_5 = (0,0,3,0,1)$, and ${\bf a}_6 = (0,0,0,3,1)$.  Take $A' = \{{\bf a}_1,{\bf a}_2\}$, so $\beta = (1,2,2,1,2)\in C(A)^\circ$.  Let $u=(u_1,\dots,u_5)\in C(A)^\circ\cap{\mathbb Z}^5$, i.~e., all $u_i$ are $>0$ and $\sum_{i=1}^4 u_i = 3u_5$.  From Equation (2.2), the series $F_{-u}(\Lambda)$ is zero unless both $u_1+u_4-u_2$ and $u_1+u_4-u_3$ are divisible by $3$, in which case the coefficients of $F_{-u}(\Lambda)$ can be expressed in terms of $l_3$ and $l_6$:
\begin{equation}
\frac{(3l_3+u_1-1)!(3l_6+u_4-1)!}{l_3!l_6!(l_3+l_6+\frac{u_1+u_4-u_2}{3})!(l_3+l_6+\frac{u_1+u_4-u_3}{3})!}
\end{equation}
These coefficients are integral for all $l_3,l_6\in{\mathbb N}$ by the above discussion.  For example, when $u=\beta$, this gives the integrality of the ratios
\[ \frac{(3l_3)!(3l_6)!}{l_3!l_6!(l_3+l_6)!^2}. \]
Taking $u=(2,1,1,2,2)$, for example, gives the integrality of the ratios
\[ \frac{(3l_3+1)!(3l_6+1)!}{l_3!l_6!(l_3+l_6+1)!^2}. \]
The first set of ratios was known to be integral by a theorem of Landau (see Example~2) but the integrality of the second set seems to be new.

{\bf Example 2.}
We generalize one direction of a classical result of Landau\cite{L} on the integrality of factorial ratios.  Let $c_{js}, d_{ks}\in{\mathbb N}$, $1\leq j\leq J$, $1\leq k\leq K$, $1\leq s\leq r$, and let
\begin{align}
C_j(x_1,\dots,x_r) &= \sum_{s=1}^r c_{js}x_s, \\
D_k(x_1,\dots,x_r) &= \sum_{s=1}^r d_{ks}x_s. 
\end{align}
To avoid trivial cases, we assume that no $C_j$ or $D_k$ is identically zero and that $C_j\neq D_k$ for all $j$ and $k$.  We also assume that for each $s$, some $c_{js}\neq 0$ or some $d_{ks}\neq 0$, i.~e., each variable $x_s$ appears in some $C_j$ or $D_k$ with nonzero coefficient.  We always make the hypothesis that
\begin{equation}
\sum_{j=1}^J C_j(x_1,\dots,x_r) = \sum_{k=1}^K D_k(x_1,\dots,x_r),
\end{equation}
i.~e.,
\begin{equation}
\sum_{j=1}^J c_{js} = \sum_{k=1}^K d_{ks} \quad\text{for $s=1,\dots,r$.}
\end{equation}
Consider the ratios
\begin{equation}
E(m_1,\dots,m_r):=\frac{\prod_{j=1}^J C_j(m_1,\dots,m_r)!}{\prod_{k=1}^K D_k(m_1,\dots,m_r)!}
\end{equation}
for $m_1,\dots,m_r\in{\mathbb N}$.  We identify these ratios as the coefficients of a hypergeometric series.

Put $n= r+J+K$.  Let ${\bf a}_1,\dots,{\bf a}_n$ be the standard unit basis vectors in ${\mathbb R}^n$ and for $s=1,\dots,r$ let
\[ {\bf a}_{n+s} = (0,\dots,0,1,0,\dots,0,c_{1s},\dots,c_{Js},-d_{1s},\dots,-d_{Ks}), \]
where the first $r$ coordinates have a $1$ in the $s$-th position and zeros elsewhere.  Our hypothesis that some $c_{js}$ or some $d_{ks}$ is nonzero implies that ${\bf a}_1,\dots,{\bf a}_{n+r}$ are all distinct.  Put $N=n+r$ and let $A = \{{\bf a}_i\}_{i=1}^{N}\subseteq {\mathbb Z}^n$.  
Put $M=r+J$ and $\beta = \sum_{j=1}^M {\bf a}_j = (1,\dots,1,0,\dots,0)$, where the ones occur in the first $M$ entries. 
From (2.2) we have
\begin{equation}
F_{-\beta}(\Lambda) = 
 \big(\Lambda_1\cdots\Lambda_M\big)^{-1}\sum_{\substack{l\in{\mathbb N}^N\\ -\sum_{j=1}^M l_j{\bf a}_j+\\ \sum_{j=M+1}^n l_j{\bf a}_j = {\bf 0}}} \frac{\prod_{j=1}^M l_j!}{\prod_{j=M+1}^N l_j!} \frac{\prod_{j=M+1}^N \Lambda_j^{l_j}}{\prod_{j=1}^M (-\Lambda_j)^{l_j}}.
\end{equation}

From the formulas for the ${\bf a}_j$ it is easy to find all solutions $l\in{\mathbb N}^N$ of the equation
\begin{equation}
-\sum_{j=1}^M l_j{\bf a}_j + \sum_{j=M+1}^N l_j{\bf a}_j = {\bf 0}
\end{equation}
(see \cite[Section 3]{AS4}).  An $N$-tuple $l\in{\mathbb N}^N$ satisfies (7.11) if and only if 
\begin{equation}
l=(m_1,\dots,m_r,C_1(m),\dots,C_J(m),D_1(m),\dots,D_K(m),m_1,\dots,m_r)
\end{equation}
for some $m=(m_1,\dots,m_r)\in{\mathbb N}^r$.  Equation (7.10) can thus be rewritten as
\begin{multline}
F_{-\beta}(\Lambda) = \\
\big(\Lambda_1\cdots\Lambda_M\big)^{-1} \sum_{m_1,\dots,m_r = 0}^\infty \frac{\prod_{j=1}^J C_j(m)!}{\prod_{k=1}^K D_k(m)!} \prod_{s=1}^r \frac{\Lambda_{n+s}^{m_s}}{(-\Lambda_s)^{m_s}}\frac{\prod_{k=1}^K \Lambda_{r+J+k}^{D_k(m)}}{\prod_{j=1}^J (-\Lambda_{r+j})^{C_j(m)}}.
\end{multline}
Thus the factorial ratios (7.9) are all integral if and only if the series $F_{-\beta}(\Lambda)$ has integral coefficients.

Landau\cite{L} has characterized the integrality of the ratios $E(m)$ (see \cite{AS4} for some additional comments on Landau's result).  Let $\lfloor r \rfloor$ denote the greatest integer less than or equal to the real number $r$.
\begin{theorem}
Assume that (7.8) holds.  One has $E(m_1,\dots,m_r)\in{\mathbb N}$ for all $m_1,\dots,m_r\in{\mathbb N}$ if and only if the step function
\begin{equation}
\Phi(x_1,\dots,x_r):=  \sum_{j=1}^J \lfloor C_j(x_1,\dots,x_r) \rfloor - \sum_{k=1}^K \lfloor D_k(x_1,\dots,x_r)\rfloor
\end{equation}
is $\geq 0$ for all $x_1,\dots,x_r\in [0,1)$.
\end{theorem}

To compare this result with Theorem 2.13, recall \cite[Theorem~2.1(a)]{AS4}: 
\begin{proposition}
Assume that (7.8) holds.  One has
\begin{equation}
M = \min\bigg\{\sum_{i=1}^n u_i \mid u=(u_1,\dots,u_n)\in C(A)^\circ\cap{\mathbb Z}^n\bigg\}
\end{equation}
if and only if $\Phi(x)\geq 0$ for all $x\in[0,1)^r$.  
\end{proposition}

Thus Landau's criterion may be restated as follows.
\begin{theorem}
Assume that (7.8) holds.  One has $E(m)\in{\mathbb N}$ for all $m\in{\mathbb N}^r$ if and only if (7.17) holds.
\end{theorem}

Under assumption (7.8) the elements of the set $A$ all lie on the hyperplane $\sum_{i=1}^n u_i =1$ in ${\mathbb R}^n$, so the remark at the end of Section~2 implies that Theorem~2.13 may be applied to this situation.  It follows from \cite[Lemma 2.5]{AS4} that $\beta\in C(A)^\circ$.  By Theorem 2.13 we then have the following result.
\begin{proposition}
Assume that (7.8) holds.  If
\[ M = \min\bigg\{\sum_{i=1}^n u_i \mid u=(u_1,\dots,u_n)\in C(A)^\circ\cap{\mathbb Z}^n\bigg\} \]
(or equivalently if $\Phi(x)\geq 0$ for all $x\in [0,1)^r$), then the series $F_u(\Lambda)$ for $u\in\big(-C(A)^\circ\big)\cap{\mathbb Z}^n$ all have integral coefficients.
\end{proposition}

We give an explicit example.  Consider the ratios
\begin{equation}
 \frac{(30m)!m!}{(15m)!(10m)!(6m)!}\quad\text{for $m\in{\mathbb N}$.} 
\end{equation}
By the above discussion they correspond to the set $A=\{{\bf a}_j\}_{j=1}^7\subseteq{\mathbb R}^6$, where ${\bf a}_1,\dots,{\bf a}_6$ are the standard unit basis vectors and
\[ {\bf a}_7 = (1,30,1,-15,-1,-6), \]
and where $\beta = (1,1,1,0,0,0)\in C(A)^\circ$.  One checks that Landau's criterion holds, hence the ratios (7.20) are integral for all $m\in{\mathbb N}$.  By Proposition~7.16 we have
\begin{equation}
3 =  \min\bigg\{\sum_{i=1}^6 u_i \mid u=(u_1,\dots,u_6)\in C(A)^\circ\cap{\mathbb Z}^6\bigg\}
\end{equation}
(this can also be verified by a direct calculation from the set $A$).  It follows from Proposition~7.19 that all the series $F_u(\lambda)$ for $u\in \big(-C(A)^\circ\big)\cap{\mathbb Z}^n$ have integral coefficients.

The choice $u=-\beta$ gives the ratios (7.20).  One can check that 
\[ u=(1,7,1,-3,-2,-1)\in \big(-C(A)^\circ\big)\cap{\mathbb Z}^n. \]
One computes the coefficients of $F_{-u}(\Lambda)$ to be
\[ \frac{(30m+6)!m!}{(15m+3)!(10m+2)!(6m+1)!}, \]
which are integral by Proposition 7.19.  As another exanple, one can check that
\[ u=(1,29,1,-14,-9,-5)\in \big(-C(A)^\circ\big)\cap{\mathbb Z}^n, \]
for which one computes the coefficients of $F_{-u}(\Lambda)$ to be 
\[ \frac{(30m+28)!m!}{(15m+14)!(10m+9)!(6m+5)!}. \]
Again, these are integral by Proposition 7.19.

\end{document}